\documentclass[reqno, 12pt]{amsart}
\usepackage{amssymb}
\usepackage{graphicx}

\theoremstyle{plain}
\newtheorem{theorem}{Theorem}[section]
\newtheorem{lemma}[theorem]{Lemma}

\newtheorem{corollary}[theorem]{Corollary}
\theoremstyle{definition}
\newtheorem{definition}[theorem]{Definition}

\newtheorem{example}[theorem]{Example}

\newtheorem{question}[theorem]{Question}

\numberwithin{equation}{section}

\makeatletter
\def\@map#1#2[#3]{\mbox{$#1 \colon #2 \longrightarrow #3$}}
\def\map#1#2{\@ifnextchar [{\@map{#1}{#2}}{\@map{#1}{#2}[#2]}}
\makeatother
\newcommand{\R}{\mathbb{R}}

\newcommand{\T}{\mathbb{T}}

\newcommand{\SI}{\ensuremath{\mathbb{S}^1}}
\newcommand{\eps}{\varepsilon}
\newcommand{\te}{\vartheta}
\renewcommand{\phi}{\varphi}

\newcommand{\fa}{f_\alpha}
\newcommand{\orb}{\mathcal{O}}
\newcommand{\SNA}{SNA}

\title{Skew Product Attractors and concavity}
\author{Llu\'{\i}s Alsed\`a}
\address{Departament de Matem\`atiques, Edifici Cc, Universitat
Aut\`onoma de Barcelona, 08913 Cerdanyola del Vall\`es, Barcelona,
Spain}
\email{alseda@mat.uab.cat}
\author{Micha\l\ Misiurewicz}
\address{Department of Mathematical Sciences, IUPUI, 402 N. Blackford
    Street, Indianapolis, IN 46202}
\email{mmisiure@math.iupui.edu}
\thanks{The first author has been partially supported by the MICINN
grant number MTM2008-01486.}
\date{May 8, 2012}

\begin{document}
\begin{abstract}
We propose an approach to the attractors of skew products that tries
to avoid unnecessary structures on the base space and rejects the
assumption on the invariance of an attractor. When nonivertible maps
in the base are allowed, one can encounter the mystery of the
vanishing attractor. In the second part of the paper, we show that if
the fiber maps are concave interval maps then contraction in the
fibers does not depend on the map in the base.
\end{abstract}
\maketitle

\section{Introduction}

We want to propose a unified approach to many situations where
attractors for skew products are considered. This includes random
systems, nonautonomous systems, strange nonchaotic attractors, etc.
(we let the reader to continue this list, warning that the terminology
may vary). This approach is built on existing ideas, but contains two
new ingredients. The first one is a realization that the space can
have various structures, so one should consider which results require
which structures. The second one is the introduction of the
possibility that an attractor is not invariant.
Moreover, admitting the possibility that
the base map is nonivertible, we encounter ``the mystery of the
vanishing attractor.'' An attractor present for an invertible map in
the base, vanishes when we forget about the past and replace the base
map by a noninvertible one. This happens in spite of the fact that the
future dynamics does not depend on the past. From the philosophical
point of view this can be interpreted as an application of Mathematics
to History: even if our future depends only on our present, making
predictions is much more consistent when we take our past into
account.\footnote{We hope that this remark will allow us to present
  this paper to the administrators in our universities as a result of
  the interdisciplinary research.}

In the second part of the paper we show how changing the point of view
allows us to pinpoint the real reasons for existence of a Strange
Nonchaotic Attractor if the fibers are one-dimensional and the maps in
the fibers are concave.

We concentrate on the discrete case, that is, on iterates of self-maps of
some space. Thus, we consider a skew product on a space $X=B\times Y$.
The space $B$ is the \emph{base} and $Y$ is the \emph{fiber space};
for each $\te\in B$ the set $\{\te\}\times Y$ is the \emph{fiber over}
$\te$. We will denote by $\pi_2:X\to Y$ the projection
$\pi_2(\te,x)=x$. The skew product $F:X\to X$ can be written as
\begin{equation}\label{skew}
F(\te,x)=(R(\te),\psi(\te,x)).
\end{equation}
Here $R$ is a map from the base to itself, and for any $\te\in B$ the
map $\psi_\te$, given by $\psi_\te(x)=\psi(\te,x)$ maps the fiber over
$\te$ to the fiber over $R(\te)$.

We are interested in the attractors. Attraction of the trajectories to
some set occurs in fibers, so we have to assume that the fiber space
$Y$ is a metric space. Now, from the point of view of the theory of
random maps (see, e.g.~\cite{Ar, CF}), what happens in the base is a
random process, which influences the maps in fibers. Thus, the system
in the base will be usually some measure-preserving transformation.
From the point of view of the theory of Strange Nonchaotic Attractors
(\emph{SNA}'s; see, e.g. a survey~\cite{JTNO}), often the topology
(and usually a differentiable structure) is added in the base, while
the invariant measure is kept there. Since those systems have usually
connections with physical models, this measure is in such cases
``natural''. The influence of what happens in the base is considered
as external forcing. From the point of view of the theory of
nonautonomous systems (this name is being used in various meanings; we
want to differentiate it from the random systems), no structure in the
base is necessary (see, e.g.~\cite{KMS}). Most often, this is just the
shift by 1 on (nonnegative) integers or a cyclic permutation on a
finite set.

Let us compare a skew product system with a ``usual'' one, that is,
with a continuous map $\Phi:Z\to Z$ of a metric space $Z$ to itself.
In both cases, we iterate the map (be it $F$ or $\Phi$) and look for
the attractors, that is, loosely speaking, sets to which trajectories
converge. However, in the usual case, this convergence is \emph{in the
whole space} $Z$, while in the skew product case the convergence is
\emph{fibrewise} (i.e. on the sets $\{\te\}\times Y$). This is the basic
difference and it has far reaching consequences.

While, especially in the applications, the phase spaces $Y$ or $Z$ are
not necessarily compact, quite often all interesting dynamics happens
in some compact subset of those spaces. Moreover, a large part of the
``pure'' theory of dynamical systems has been built under the
assumption that the phase space is compact. Therefore in the sequel we
will assume that \emph{$Y$ is a compact metric space}. Moreover, in
order to be able to state our ideas more clearly, at a certain moment
we will start to assume that $Y$ is a closed interval.

A remarkable paper on SNA's in this setting is the one by Keller
\cite{Keller} where a deep study on the existence and properties of
the attractors is conducted for systems on $\SI \times [0,+\infty)$,
of the form
\begin{equation}\label{SOri}
T\begin{pmatrix}\theta\\x \end{pmatrix} =
\begin{pmatrix} R_\omega(\theta)\\ p(x)q(\theta)\end{pmatrix},
\end{equation}
where $R_\omega(\theta) = \theta + \omega \pmod{1}$ is an irrational
rotation of the circle $\SI=\R/\T$, {\map{q}{\SI}[{[0,+\infty)}]}
is a continuous function and {\map{p}{[0,+\infty)}} is a continuous,
bounded strictly increasing and strictly concave function.
A lot of subsequent analytical studies on the attractors of similar
systems rely on the concavity of the function $p$ on the fibers
(\cite{AlsMis, Haro, Bjerk}) and some other studies deal with
the relation between the monotonicity of this function and the
existence of \SNA's (\cite{AlsMis, Bjerk}). We should also mention
papers \cite{Bjerk1, Bjerk2, Jager}, where certain techniques for
proving the existence of a SNA when the fiber is one-dimensional, were
developed.

When we analyze this system from our new point of view, we notice that
the existence of a SNA is caused solely by the concavity properties of
the maps in the fibers, and the use of the ergodic theory tools for
the basic results is practically unnecessary. Attraction (or
contraction) in the fibers is the result of concavity, and not of the
averaging.

The paper is organized as follows. In Section~\ref{sec-spa} we present
our general philosophy. Then in Section~\ref{sec-noninv} we study the
situation when a skew product has a noninvertible base map. We present
an example of a simple system for which a function from the base to
the fibers whose graph is an attractor exists, but such function
cannot be Borel or measurable for any invariant measure with positive
entropy. On the other hand, when you replace the system in the base by
its natural extension, there is a very regular (continuous,
measurable) function whose graph is an attractor. This attractor
vanishes when we forget about the past, although attraction is defined
by looking at the future behavior, so we see the mystery of the
vanishing attractor, mentioned earlier.

In the rest of the paper we consider systems with interval fibers and
concave maps on them. In Section~\ref{sec-ns} we forget about the
structure of the base space and we consider our skew product as a
bunch of full orbits. Then, by using the estimates and notions
introduced in Section~\ref{sec-ac}, we prove that for nonautonomous
dynamical systems concavity always implies contraction on the fibers
both when the fiber maps are monotone and nonmonotone.

In Section~\ref{sec-nea} we use the results from the preceding section
in the setting of skew products to study the influence of concavity on
the existence of attractors, their basin of attraction and the
positiveness almost everywhere of these attractors. We consider skew
products similar to the Keller ones, that is, where the fibers are
intervals of the form $[0,a]$ with $a > 0$ and the fiber maps
$\psi_\theta$ are such that $\psi_\theta(0) = 0$ for every $\theta$ in
the base $B$. Then, the set $B \times \{0\}$ is invariant. As usual,
the \emph{pinched set} is defined as the set of all points in the
fibers whose image is contained in $B \times \{0\}$ (and, due to the
invariance of this set, stay in $B \times \{0\}$ in all subsequent
iterates). There are three main conclusions of this section. First, we
show that concavity implies the existence of an attractor which is the
graph of a function from the base $B$ to the fiber space and the basin
of attraction of this attractor contains all points whose forward
trajectory is nonpinched. However, these attractors are not
necessarily invariant. Finally, we show that when additionally the
base map is invertible and preserves an ergodic invariant measure
$\mu$, then the function whose graph is the attractor is either 0
$\mu$-almost everywhere or positive $\mu$-almost everywhere. In the
latter case, the attractor is invariant and its basin of attraction is
of the form $Z \times (0,a]$, where $Z \subset B$ has full
$\mu$-measure.

\section{Skew Product Attractors}\label{sec-spa}

We consider the skew product~\eqref{skew}, with $Y$ a compact metric
space. As we mentioned in the introduction, we consider only
attraction in the fibers. That is, the distance of a point $p$ from
the attractor is measured as the distance of $p$ from the intersection
of the attractor and the fiber containing $p$. Consequently,
attraction means that this distance goes to zero.

The intersection of an attractor with a fiber should look like an
attractor for a usual system. In particular, it should be compact.
However, requiring its invariance is not a natural thing to do,
because (with obvious exceptions) the image of a fiber is a different
fiber. We will discuss this problem later. Normally for an attractor
$A$ one defines its \emph{basin of attraction} as the set of points
$p$ such that the distance of the $n$-th image of $p$ from $A$ goes to
0 as $n\to\infty$. Here we can repeat this definition, except that,
as said earlier, here ``distance'' means distance from the intersection of $A$ with the fiber
to which $F^n(p)$ belongs. Of course the basin of attraction of an
attractor should contain this attractor (in the usual case this
follows from invariance), and should be in some sense big. This
motivates the following definition.

\begin{definition}
A set $A\subset X$ is an \emph{attractor} of $F$ if
\begin{enumerate}
\item the intersection of $A$ with every fiber is compact;
\item the basin of $A$ contains $A$;
\item the basin of $A$ intersected with every fiber contains a
neighborhood of $A$ (topological attractor) or has positive measure
(measure attractor).
\end{enumerate}
\end{definition}

We have to add some comments. When we speak of positive measure, we
assume that there is some (natural) measure in the fiber space $Y$
that we use. The second comment is that instead of requiring that the
properties above hold for all fibers, we may require that they hold
for almost all fibers (here we assume that the base map is a measure
preserving transformation).

The next comment is about invariance. For the usual dynamical system,
an attractor is compact, and therefore it contains the $\omega$-limit
set of every point from its basin. Thus, we can replace it, if
necessary, by its subset defined as the closure of the union of the
$\omega$-limit sets of the points from its basin. This subset is
automatically invariant. Thus, the requirement that the attractor is
invariant is natural and in a sense, is satisfied automatically. In
the skew product case it is not even clear how the $\omega$-limit set
should be defined if we apply our fiberwise approach. Indeed,
(except in the periodic fibers, which often do
not exist or their union has measure zero) the distance between the
points is measured only in the same fiber, and the trajectory does not
visit the same fiber twice. Thus, there is no special reason to
require the attractor to be invariant.

In many cases a \emph{pullback attractor} is considered (see,
e.g.~\cite{CF, Sch}). Then, instead of taking a point and its forward
orbit, one goes back in time with an iterate of the base map, takes a
point there, and returns with the same iterate of the map $F$. While
we agree that this notion is very useful in many cases, we cannot
agree with the opinion presented in~\cite{Sch} that this notion is
better and more natural than the notion of the usual forward
attractor. The past may not exist. It may exist but be not unique. And
even if it exists and is unique, the aim of considering a dynamical
system is to try to predict the future from our knowledge of the
present, rather than to predict the present from our knowledge of the
past.

Now we pass to the special case that we investigate closer. Namely, we
assume that $Y$ is a closed interval, and the intersection of the
attractor with each fiber consists of one point. Then the whole
attractor is the graph of a function $\phi:B\to Y$. We assume also
that the attractor is almost global, that is, for every (or almost
every) fiber the intersection of its basin with the fiber contains all
points of the fiber except perhaps the endpoints.

We will finish this section by showing why it is a good idea to have a
measure structure in $B$. The first reason is the following theorem.

\begin{theorem}\label{unique}
Assume that for a skew product~\eqref{skew} there is an ergodic
invariant measure $\mu$ for $R$ on the base $B$. Then, if the graphs
of measurable functions $\phi_1,\phi_2:B\to Y$ are both attractors, it
follows that $\phi_1=\phi_2$ $\mu$-almost everywhere.
\end{theorem}

\begin{proof}
Suppose that $\phi_1$ and $\phi_2$ are not equal $\mu$-almost
everywhere. Then there exists $\eps>0$ and a measurable set $Z\subset
B$ of positive measure, such that $d(\phi_1(\te),\phi_2(\te))>\eps$
for every $\te\in Z$, where $d$ is the metric in $Y$. Since $\mu$ is
ergodic, almost every trajectory of $R$ visits $Z$ infinitely often.
Therefore, for almost every $\te\in B$ and every $x\in Y$ the maximum
of the distances of $\pi_2(F^n(\te,x))$ from $\phi_1(R^n(\te))$ and
from $\phi_2(R^n(\te))$ is larger than $\eps/2$ for infinitely many
$n$'s. This means that it is impossible for the graphs of both
$\phi_1$ and $\phi_2$ to be attractors.
\end{proof}

Now we present an example what can go wrong if we want to get an
attractor everywhere instead of almost everywhere.

\begin{example}\label{noinvattr}
Let $B=\{\te_n\}_{n=-\infty}^\infty\cup\{-1,1\}$, where
$\te_n=1-\frac1{n+1}$ if $n\ge 0$ and $\te_n=-1-\frac1n$ if $n<0$. The
map $R$ fixes $-1$ and $1$, and maps $\te_n$ to $\te_{n+1}$. We
define $\psi_\te(x)$ to be $x(2-x)$ if $\te\ge 0$ and
$x(2-x)/4$ if $\te<0$ (see Figure~\ref{noattr}). Assume that the graph
of a function $\phi:B\to[0,1]$ is an invariant attractor. Since
$x(2-x)/4\le x/2$, we have $\phi(\te_0)\le \phi(\te_{-n})/2^n \le
1/2^n$ for all positive $n$. Thus, $\phi(\te_0)=0$, and consequently,
$\phi(\te_n)=0$ for all $n>0$. On the other hand, the trajectory of
every $x\in (0,1]$ under the map $x\mapsto x(2-x)$ goes to 1, so
$\phi(\te_n)\to 1$ as $n\to\infty$. This is a contradiction, and
therefore in this case there is no $\phi$ whose graph is an invariant
attractor.
\end{example}

\begin{figure}
\begin{center}
\includegraphics[height=50truemm]{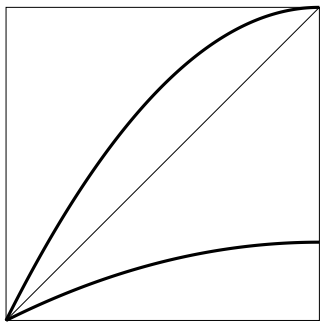}
\caption{Maps $x\mapsto x(2-x)$ and $x\mapsto x(2-x)/4$ on
  $[0,1]$.}\label{noattr}
\end{center}
\end{figure}

\section{Noninvertible base map}\label{sec-noninv}

In this section we consider a model which does not satisfy our
assumption from the preceding section (in particular, the fiber space
consists of only two points), but illustrates very well the problems
that we can encounter when considering a noninvertible map in the
base. It has a great advantage of being simple, and can be interpreted
as flipping a coin. It leads us to the mystery of the vanishing
attractor.

In the base $B$ we take the full one-sided shift $R$ on 2 symbols (0
and 1). The fiber space consists of two points (again 0 and 1). We
will use the notation $x=(x_0,x_1,x_2,\dots)\in B$ with
$x_i\in\{0,1\}$, so $R(x)=(x_1,x_2,x_3,\dots)$. The map
$F:B\times\{0,1\}\to B\times\{0,1\}$ is given by
\[
F(x,y)=(R(x),x_0).
\]

In this setup, the graph of a function $\phi:B\to\{0,1\}$ is an
attractor if and only if for every $(x,y)\in B\times\{0,1\}$ there
is $N$ such that for every $n\ge N$
\begin{equation}\label{m1}
\pi_2(F^n(x,y))=\phi(R^n(x)).
\end{equation}

We will show first that such a function exists.

\begin{definition}
Consider the following equivalence relation in $B$: the points $\te$
and $\sigma$ are equivalent if $R^n(\te)=R^m(\sigma)$ for some
nonnegative integers $n,m$. The equivalence classes of this relation
will be called the \emph{full orbits} of $R$.
\end{definition}

\begin{theorem}\label{nbm}
For the system defined above, there exists a function
$\phi:B\to\{0,1\}$ whose graph is an attractor with the basin of
attraction equal to the whole space.
\end{theorem}

\begin{proof}
We can look at $B$ as the disjoint union of full orbits of $R$. For
each such orbit $\orb$, we choose one element $x=(x_0,x_1,x_2,\dots)
\in\orb$. Then for each $n\ge 1$ we set $\phi(R^n(x))=x_{n-1}$.
Observe that in such a way for every $y\in\{0,1\}$ we have
$F(R^{n-1}(x),y)=(R^n(x),\phi(R^n(x))$. Now, for every $z\in\orb$, if
$m$ is sufficiently large, there exists $n>0$ such that
$R^{m-1}(z)=R^{n-1}(x)$, so for every $y\in\{0,1\}$ we have
$\pi_2(F(R^m(z),y)=\phi(R^m(z))$. Thus, if we define $\phi$ in an
arbitrary way at the remaining points of $\orb$, and make this
construction for all full orbits of $R$, then the graph of $\phi$ will
be a global attractor.
\end{proof}

If $\mu$ is an ergodic invariant measure for $R$, then the question is
whether a function $\phi$, whose graph is an attractor, can be
measurable.

\begin{theorem}\label{model}
For the system described above and an ergodic invariant probability
measure $\mu$ on $B$, if there exists a $\mu$-measurable function
$\phi:B\to\{0,1\}$ whose graph is an attractor with the basin of
attraction $Z\times \{0,1\}$ for some set $Z\subset B$ of $\mu$
measure $1$, then the entropy of $\mu$ is zero.
\end{theorem}

\begin{proof}
Set $A=\{x\in B:x_0=\phi(R(x))\}$ and $C=\bigcap_{n=0}^\infty
R^{-n}(A)$. With this notation, if $n\ge 1$ then
\[
\pi_2(F^n(x,y))=(R^{n-1}(x))_0,
\]
so~\eqref{m1} is
equivalent to $R^{n-1}(x)\in A$. Thus, the graph of $\phi$ is an
attractor on a set of full measure if and only if for almost every
$x\in B$ there is $N\ge 1$ such that $R^N(x)\in
C$, that is, if
\[
\mu\left(\bigcup_{N=1}^\infty R^{-N}(C)\right)=1.
\]
Note also that $\sigma(C)\subset C$.

Clearly, if $\phi$ is measurable then $A$ is measurable, so $C$ is
also measurable. Since $C$ is invariant and $\mu$ is ergodic, $C$ has
measure 0 or 1. If it has measure 0, the union of its preimages has
measure 0, a contradiction. This proves that $\mu(C)=1$. Since
$C\subset A$, we get $\mu(A)=1$.

Let $\xi:B\to B$ be the map that replaces $x_0$ by $1-x_0$. By the
definition, at most one of the points $x,\xi(x)$ can belong to $A$.
Therefore, $A\cap\xi(A)=\emptyset$. This means that the shift $R$ is
one-to-one $\mu$-almost everywhere. Since $R$ has a one-sided
generator, this implies that $h_\mu(R)=0$ (see, e.g., \cite{Parry}).
\end{proof}

Since there exist ergodic invariant probability measures for $R$ with
positive entropy, we get the following corollary.

\begin{corollary}\label{noBorel}
For the system described above, there is no Borel function
$\phi:B\to\{0,1\}$ whose graph is an attractor.
\end{corollary}

Consider now what happens if instead of one-sided shift in the base we
consider its natural extension, the two-sided shift. In this case the
graph of the map $\psi$, given by
\[
\psi(\dots,x_{-2},x_{-1},x_0,x_1,x_2,\dots)=x_{-1}
\]
is clearly an attractor with the whole space as the basin of
attraction (because this graph is the image of the whole space). This
attractor is invariant, the function $\psi$ is continuous, and
therefore measurable for all invariant measures. To summarize -- this
is the best attractor one can dream of. Yet when we return to our
original system (which can be interpreted as forgetting about the
past), this attractor vanishes. This is a paradox, because attractors
are defined by looking forward in time, so why should forgetting of
the past have any influence on them? We call this strange phenomenon
``the mystery of the vanishing attractor.''

\section{$\alpha$-concavity}\label{sec-ac}

In this section we introduce the notion of $\alpha$-concavity and
we obtain estimates that relate the $\alpha$-concavity with the
contraction of an interval map, both in the monotone and nonmonotone
case. These estimates will be useful later.

\begin{definition}
Let $f$ be a continuous real-valued function on a closed interval $I$
of the real line and let $\alpha\ge 0$. The function $f$ will be
called \emph{$\alpha$-concave} if the function $\fa$, given by
\[
\fa(x)=f(x)+\alpha x^2
\]
is concave.
\end{definition}

The following properties of an $\alpha$-concave function $f$ follow
immediately from the definition:
\begin{enumerate}
\item $f$ is concave;
\item if $\alpha>0$ then $f$ is strictly concave;
\item if $0\le\beta\le\alpha$ then $f$ is $\beta$-concave.
\end{enumerate}

Now we will prove two inequalities satisfied by $\alpha$-concave
functions that we will need later.

Assume that the left endpoint of $I$ is $0$. Given two points $u,v >
0$ we define
\[
 \kappa(u,v) := \frac{|v-u|}{\min\{u,v\}}.
\]

\begin{lemma}\label{fratio-inc}
Assume that $f$ is $\alpha$-concave in the interval $[0,y]$ and
$f(0) = 0 < f(y)$. Let $x \in (0,y)$ be such that
$0 < f(x) < f(y).$ Then,
\[
 \frac{\kappa(f(x) , f(y))}{\kappa(x, y)} \le
 \frac{f(y)}{f(y)+\alpha y^2}.
\]
\end{lemma}

\begin{proof}
By concavity of $\fa$ we have $\fa(x)/x\ge\fa(y)/y$. Therefore
\begin{equation}\label{eq-fi}
\frac{f(x)}x\ge\frac{f(y)}y+\alpha(y-x),
\end{equation}
so
\[
f(x)\ge\frac{xf(y)}y+\alpha x(y-x).
\]
Thus,
\[
f(y)-f(x)\le\frac{yf(y)}y-\frac{xf(y)}y-\alpha x(y-x)=
(y-x)\left(\frac{f(y)}y-\alpha x\right).
\]
{}From this and~\eqref{eq-fi} we get
\begin{align*}
\frac{\kappa(f(x),f(y))}{\kappa(x,y)}
&= \frac{f(y) - f(x)}{f(x)}\cdot \frac{x}{y-x}
\le \frac{\tfrac{f(y)}{y}-\alpha x}{\tfrac{f(y)}{y}+\alpha(y-x)}\\
&= 1 - \frac{\alpha y}{\tfrac{f(y)}{y} + \alpha(y-x)}
\le 1 - \frac{\alpha y}{\tfrac{f(y)}{y} + \alpha y}
= \frac{\tfrac{f(y)}{y}}{\tfrac{f(y)}{y} + \alpha y}
= \frac{f(y)}{f(y) + \alpha y^2}.
\end{align*}
\end{proof}

Recall that for a concave map the one-sided derivatives are well
defined. We will denote the left one-sided derivative of $f$ by
$f'_-.$

Let $f$ be a strictly concave nonnegative function on the interval
$[0,a]$, with $f(0)=0$. Observe that there exists a unique point
$c\in[0,a]$ such that $f(c) = \max\{f(x) \colon x \in [0,a]\}$, that
is, $f$ is strictly increasing on $[0,c]$ and strictly decreasing on
$[c,a]$ (but note that it may happen that $c=a$).

By strict concavity, for every $x \in (0,c]$ we have $f'_-(x) <
f(x)/x$. Usually this inequality, with the absolute value of the
derivative, can be extended further to the right of $c$. Set
\begin{equation}\label{definb}
b=\sup\left\{x\in [0,a]\colon |f'_-(x)|<\frac{f(x)}{x}\right\}.
\end{equation}
Since often the absolute values of the slopes of the tangent line to
the graph of $f$ at $b$ and the line joining $(0,0)$ with $(b,f(b))$
are equal, we will call $b$ the \emph{isoclinic point} of $f$.

Note that $c\le b \le a$ and $f(b) > 0$. However, we can prove more.

\begin{lemma}\label{isoclinic}
Let $f$ be a strictly concave nonnegative map of an interval $[0,a]$
to itself, with $f(0)=0$, whose isoclinic point is $b$. Then $b\ge
a/2$.
\end{lemma}

\begin{proof}
Suppose that $b<a/2$. Then $c<a/2<a$, so by strict concavity,
\[
|f'_-(a/2)|<\frac{f(a/2)-f(a)}{a-a/2}\le\frac{f(a/2)}{a/2}.
\]
Thus, by the definition, $a/2\le b$.
\end{proof}

\begin{lemma}\label{fratio-dec}
Assume that $f$ is $\alpha$-concave in the interval $[0,a]$ with
$f(0) = 0$ and $\alpha>0$. Let $x,y \in (0,b)$ be such that
$x < y$ and $0 < f(y) < f(x)$. Then
\begin{equation}\label{eq-fd}
\frac{\kappa(f(x),f(y))}{\kappa(x, y)} <
1-\frac{\alpha b (b-x)}{f(b)}.
\end{equation}
\end{lemma}

\begin{proof}
With our assumptions we have
\begin{equation}\label{eq-fd0}
\frac{\kappa(f(x) , f(y))}{\kappa(x, y)}
= \frac{f(x) - f(y)}{f(y)}\cdot \frac{x}{y-x}.
\end{equation}
The assumptions $x < y$ and $f(y) < f(x)$ imply that $y > c$.
Hence, $f(b) < f(y)$ because $y < b$, so
\begin{equation}\label{eq-fd1}
\frac{x}{f(y)}<\frac{b}{f(b)}.
\end{equation}
Moreover, by the $\alpha$-concavity of $f$ we get
\begin{equation}\label{eq-fd2}
\begin{split}
\frac{f(x)-f(y)}{y-x} &\le\frac{f(x)-f(b)}{b-x}\\
&=\frac{\fa(x)-\alpha x^2-\fa(b)+\alpha b^2}{b-x}
=\frac{\fa(x)-\fa(b)}{b-x}+\alpha(b+x).
\end{split}
\end{equation}

The left one-sided derivative of $f$ is continuous from the left, and
therefore
\[
-f'_-(b) = |f'_-(b)|\le \frac{f(b)}{b}.
\]
By this and concavity of $f_\alpha$ we get
\[
\frac{\fa(x)-\fa(b)}{b-x}\le-(\fa)'_-(b)=-f'_-(b)-2\alpha b\le
\frac{f(b)}b-2\alpha b.
\]
Together with~\eqref{eq-fd2}, this gives
\[
\frac{f(x)-f(y)}{y-x}\le\frac{f(b)}b-\alpha(b-x).
\]
This inequality together with~\eqref{eq-fd0} and~\eqref{eq-fd1}
implies~\eqref{eq-fd}.
\end{proof}

\section{Nonautonomous systems}\label{sec-ns}

When we forget about the structure of the base space, our skew product
becomes a bunch of full orbits. With this in mind, in this section we
study the relation between concavity and contraction on the fibers for
nonautonomous dynamical systems by using the estimates and notions
introduced in Section~\ref{sec-ac}. To this end we introduce the
notion of \emph{equiconcavity} which makes the notion of
$\alpha$-concavity independent of the ``scale'' of each map in the
system (that is, independent on the supremum of each of these maps).

\begin{definition}
Let $(f_n)_{n=1}^{\infty}$ be a sequence of maps from the interval
$[0,a]$ to itself such that $f_n(0)=0$ for every $n$. Such a sequence
will be called \emph{pinched} when there exists an $n$ such that $f_n$
is identically zero. Also, it will be called \emph{equiconcave} if
there exists a positive constant $\beta$ such that each $f_n$ is
$\beta\gamma_n$-concave, where $\gamma_n$ is the supremum of $f_n$.
\end{definition}

We will consider a nonautonomous dynamical system given by a
sequence $(f_n)_{n=1}^{\infty}$. That is, we apply first $f_1$, then
$f_2$, etc. We will use the standard notation for the trajectories,
that is, if the starting point with index 0 belongs to $[0,a]$, then
we define by induction the point with index $n$ as the image under
$f_n$ of the point with index $n-1$. For instance, $x_1=f_1(x_0)$,
$x_2= f_2(x_1)$, etc.

Let us consider first the case when all maps $f_n$ are nondecreasing.

\begin{theorem}\label{ConcAttr}
Let $(f_n)_{n=1}^{\infty}$ be a sequence of monotone maps from the
interval $[0,a]$ to itself such that $f_n(0)=0$ for every $n$. Assume
also that this sequence is either pinched or equiconcave. Then for
every $x_0,y_0 \in (0,a]$ we have
\begin{equation}\label{tca}
\lim_{n\to\infty} |x_n -y_n| = 0.
\end{equation}
\end{theorem}

\begin{proof}
Assume first that $(f_n)_{n=1}^{\infty}$ is pinched and let $N$ be the
smallest positive integer such that $f_N$ is identically zero. Then
$x_N = y_N = 0$, and therefore $x_n = y_n = 0$ for every $n \ge N$.
This completes the proof of the theorem in this case.

Suppose now that $(f_n)_{n=1}^{\infty}$ is equiconcave. Without loss
of generality we may assume that $0 < x_0 < y_0$, and therefore $0 <
x_n < y_n$ for every $n$. Then we have
\begin{equation}\label{prod}
|x_n-y_n|= x_n\kappa(x_n,y_n)< a\kappa(x_n,y_n)=a\kappa(x_0,y_0)
\prod_{k=0}^{n-1}\frac{\kappa(x_{k+1},y_{k+1})}{\kappa(x_k,y_k)}.
\end{equation}
By Lemma~\ref{fratio-inc} and since $\gamma_{k+1}\ge y_{k+1}$, we have
\[
\frac{\kappa(x_{k+1},y_{k+1})}{\kappa(x_k,y_k)}\le\frac{y_{k+1}}
{y_{k+1}+\beta\gamma_{k+1}y_k^2}\le\frac1{1+\beta y_k^2}.
\]

If there is $\eps>0$ such that $y_k\ge\eps$ for infinitely many
indices $k$, then by the above estimate infinitely many terms
$\kappa(x_{k+1},y_{k+1})/\kappa(x_k,y_k)$ are bounded from above by
$1/(1+\beta\eps^2)$, which is smaller than 1. Taking into account that
by Lemma~\ref{fratio-inc} $\kappa(x_{k+1},y_{k+1})/\kappa(x_k,y_k)<1$
for all $k$, we see that in this case the product in~\eqref{prod} goes
to 0 as $n\to\infty$. Thus, by~\eqref{prod}, \eqref{tca} holds.

If there is no such $\eps$, then by the definition of the limit,
$\lim_{n\to\infty}y_n=0$, and since $0<x_n<y_n$, \eqref{tca} also holds.
\end{proof}

Now we discard the assumption of monotonicity of the maps $f_n$.
However, we need some bound on the points $x_n$ and $y_n$. In a
general case we cannot expect any contraction, as the simple example
of $f_n(x)=4x(1-x)$ and $a=1$ shows. Let $b_n$ be the isoclinic point
of $f_n$.

The following theorem is a generalization of Theorem~\ref{ConcAttr}.

\begin{theorem}\label{CAnonmonot}
Let $(f_n)_{n=1}^{\infty}$ be a sequence of maps from the interval
$[0,a]$ to itself such that $f_n(0)=0$ for every $n$. Assume also that
this sequence is either pinched or equiconcave. Then for every
$x_0,y_0 \in (0,a]$ such that $x_n,y_n<b_n$ for all $n$, we have
\begin{equation}\label{tca1}
\lim_{n\to\infty} |x_n -y_n| = 0.
\end{equation}
\end{theorem}

\begin{proof}
In the pinched case the proof is exactly the same as for
Theorem~\ref{ConcAttr}. In the equiconcave case, it is very similar,
so we only point out the differences.

The first difference is that it may happen that $x_N=y_N$ for some
$n$, but then $x_n=y_n$ for all $n>N$, so~\eqref{tca1} holds.

The second difference is that it may happen that $x_k>y_k$; then we
just switch the roles of $x$ and $y$.

The third, and most important difference is that it may happen that
$x_k<y_k$ but $x_{k+1}>y_{k+1}$ (or $x_k>y_k$ but $x_{k+1}<y_{k+1}$,
then again we switch the roles of $x$ and $y$). In this case, instead
of Lemma~\ref{fratio-inc} we use Lemma~\ref{fratio-dec}. We get,
taking into account that $f(b_k)\le\gamma_k$ and
Lemma~\ref{isoclinic},
\begin{equation}\label{closetob}
\frac{\kappa(x_{k+1},y_{k+1})}{\kappa(x_k,y_k)}\le 1-\frac{\beta
\gamma_kb_k(b_k-x_k)}{f(b_k)}\le 1-\frac{\beta a(b_k-x_k)}2.
\end{equation}

Now, instead of checking whether there are infinitely many indices $k$
for which $y_k\ge\eps$, we check whether there are infinitely many
indices $k$ for which either $x_{k+1}<y_{k+1}$ and $y_k\ge\eps$ or
$x_{k+1}>y_{k+1}$ and $b_k-x_k\le\eps$ (we assume that $x_k<y_k$).
If yes, we get~\eqref{tca1} in the same way as in the proof of
Theorem~\ref{ConcAttr}, but taking additionally into account
inequality~\eqref{closetob}. If there is no $\eps>0$ for which this is
true, then the sequence $(x_k-y_k)$ splits into two subsequences. On
one of them both $x_k$ and $y_k$ go to 0, so $x_k-y_k\to 0$; on the
other one both $b_k-x_k$ and $b_k-y_k$ go to 0, so also $x_k-y_k\to
0$.
\end{proof}

\section{Nonchaotic Equiconcave Attractors}\label{sec-nea}

Let us now analyze how the results of Section~\ref{sec-ns} fit into
the scheme presented in Section~\ref{sec-spa}.

We use the notation from Section~\ref{sec-spa}.

\begin{definition}\label{eqcfam}
If $Y=[0,a]$, the family $\{\psi_\te\}_{\te\in B}$ will be called
\emph{equiconcave} if there exists a positive constant $\beta$ such
that each $\psi_\te$ is $\beta\gamma_\te$-concave, where $\gamma_\te$
is the supremum of $\psi_\te$. Note that now we included the pinched
case in the definition of equiconcavity. Indeed, if $\psi_\te$ is
identically 0 then $\gamma_\te=0$ and $\psi_\te$ is 0-concave.
\end{definition}

\begin{definition}\label{eqcskew}
If $Y=[0,a]$ and the family $\{\psi_\te\}_{\te\in B}$ satisfies
$\psi_\te(0)=0$ for each $\te\in B$ and is equiconcave, then we will
call the system $(X,F)$ an \emph{equiconcave skew product}. If
additionally all functions $\psi_\te$ are monotone, the system will be
called a \emph{monotone equiconcave skew product}. If we replace the
assumption of monotonicity by the assumption that
\[
\psi_\te([0,a])\subset [0,b_{R(\te)})
\]
for all $\te\in B$, where $b_{R(\te)}$ is the isoclinic point of
$\psi_{R(\te)}$, then the system will be called an \emph{isoclinic
equiconcave skew product}. Remember that the isoclinic point of a
monotone function is $a$. Hence, every monotone equiconcave skew
product is an isoclinic equiconcave skew product.
\end{definition}

In many standard examples of systems with strange nonchaotic
attractors one defines $\psi$ as a product: $\psi(\te,x)=f(x) g(\te)$.
In those examples, if $f$ is $\alpha$-concave for some $\alpha>0$ then
$\{\psi_\te\}_{\te\in B}$ is equiconcave and the isoclinic point for
all $\psi_\te$ is the same as for $f$. In particular, if additionally
$f$ is monotone, then we get a monotone equiconcave skew product, and
if the product of the maxima of $f$ and $g$ is smaller than the
isoclinic point for $f$ then we get an isoclinic equiconcave skew
product.

To state the results in a short way, we need more definitions.

\begin{definition}\label{preinv-def}
A graph of a function $\phi:B\to [0,a]$ will be called
\emph{preinvariant} if for every $\te\in B$ there exists $N$ such that
for every $n\ge N$ we have
\begin{equation}\label{preinv}
F(R^n(\te),\phi(R^n(\te)))=(R^{n+1}(\te),\phi(R^{n+1}(\te))).
\end{equation}
A point $\te\in B$ will be called \emph{pinching} if there are
infinitely many positive integers $n$ such that $\psi_{R^n(\te)}$ is
identically equal to 0.
\end{definition}

First we do not endow $B$ with any extra structure.

\begin{theorem}\label{preinvattr}
Let the system $(X,F)$ with base $B$ and fiber space $[0,a]$ be an
isoclinic equiconcave skew product and let $\phi:B\to [0,a]$ be a
preinvariant function, positive at any point that is not pinching.
Then the graph of $\phi$ is an attractor with the basin of attraction
containing all points whose forward trajectory does not pass through
$B\times \{0\}$.
\end{theorem}

\begin{proof}
Let $(\te,x)$ be a point whose forward trajectory does not pass
through $B\times \{0\}$. By the definition of preinvariance, there is
$N$ such that~\eqref{preinv} holds for every $n\ge N$. None of the
points $R^n(\te)$, $n=0,1,2,\dots$, is pinched, so $\pi_2(F^N(\te,x))$
is positive. By the assumptions on $(\te,x)$, $\phi(R^N(\te))$ is also
positive. Therefore, by Theorem~\ref{CAnonmonot}, the distance between
$\pi_2(F^n(\te,x))$ and $\phi(R^n(\te))$ goes to 0 as $n\to\infty$.
\end{proof}

In the proof of the next theorem, we will use a similar method as in
the proof of Theorem~\ref{nbm}.

\begin{theorem}\label{nostructpi}
Let the system $(X,F)$ with base $B$ and fiber space $[0,a]$ be an
isoclinic equiconcave skew product. Then there exists a preinvariant
function $\phi:B\to [0,a]$, positive at any point that is not pinched.
\end{theorem}

\begin{proof}
Since we do not require any special behavior of $\phi$ across the
fibers, it is enough to define $\phi$ on each full orbit $\orb$ of
$R$. We will consider three possible cases.

The first case is when there is $\te\in\orb$ which is pinched. Then
every element of $\orb$ is pinched. In this case we set $\phi$
identically 0 on $\orb$.

The second case is when no element of $\orb$ is pinched and $\orb$ is
neither periodic nor preperiodic. Then we fix $\te_0\in\orb$ such that
for $n=0,1,2,\dots$ the map $\psi_{R^n(\te_0)}$ is not identically 0,
choose some value $a_0\in (0,a]$ and set
\begin{equation}\label{pr}
\phi(R^n(\te_0))=\pi_2(F^n(\te_0,a_0))
\end{equation}
for $n=0,1,2,\dots$. At all other points of $\orb$ we set arbitrary
positive values of $\phi$.

The third case is when no element of $\orb$ is pinched and $\orb$ is
periodic or preperiodic. Then we fix $\te_0\in\orb$ which is periodic
for $R$. Let $k$ be its period. Then $F^k$ restricted to the fiber
over $\te_0$ has a fixed point $(\te_0,0)$. If it has another fixed
point $(\te_0,x)$, we choose this $x$ as $a_0$. Otherwise, we set
$a_0=0$. Then we set~\eqref{pr} for $n=0,1,2,\dots,k-1$. At all other
points of $\orb$ we set arbitrary positive values of $\phi$.

It is clear that the function $\phi$ constructed in the way described
above is preinvariant and is positive at any point that is not
pinched.
\end{proof}

From the above two theorems we obtain immediately the following
corollary.

\begin{corollary}\label{nostructattr}
Let the system $(X,F)$ with base $B$ and fiber space $[0,a]$ be an
isoclinic equiconcave skew product. Then there exists a function
$\phi:B\to [0,a]$, whose graph is an attractor with the basin of
attraction containing all points whose forward trajectory does not
pass through $B\times \{0\}$.
\end{corollary}

Observe that using this method we cannot always make an attractor
invariant. If $\phi$ is defined at $R(\te)$, we can try to define
$\phi(\te)$ as $\psi_\te^{-1}(\phi(R(\te)))$, but it may happen that
the image of $[0,a]$ under $\psi_\te$ does not contain
$\phi(R(\te))$.

Note that in general at this stage we do not have any uniqueness of
the attractor. Indeed, our choices in the construction were to a great
degree arbitrary.

Now we introduce topology. We assume that $B$ is a compact metric
space and that $F$ is continuous. We can ask whether the function
$\phi$, whose graph is an attractor, can be chosen continuous,
semicontinuous, Borel, etc. As we will see later, if $R$ is not a
homeomorphism then the situation may be very complicated. Similarly,
the general isoclinic equiconcave case is more complicated than the
monotone equiconcave one (see~\cite{AlsMis}). Thus, for simplicity let
us restrict our attention to a monotone equiconcave skew product with
a homeomorphism in the base. This makes sense, since we are presenting
mainly counterexamples.

Example~\ref{noinvattr} shows that even with those strong assumptions
we cannot count on getting an invariant attractor. However, in this
example a preinvariant attracting graph of a continuous function of
course exists; just take $\phi$ identically equal to 1.

On the other hand, if $(B,R)$ is the circle with an irrational
rotation and one of the maps $\psi_\te$ is identically equal to 0,
then a continuous function $\phi$ whose graph is an attractor has to
be equal 0 on a dense subset of a circle, so it has to be identically
0. However, one can easily produce examples where such $\phi$ is not
an attractor (see~\cite{Keller}). This means that requiring $\phi$ to be
continuous will not work.

\begin{question}
Let $(X,F)$ with base $B$ and fiber space $[0,a]$ be a monotone
equiconcave skew product, for which the base map is a homeomorphism.
Does it follow that there exists a Borel function $\phi:B\to [0,a]$,
whose graph is an attractor with the basin of attraction containing
all points whose forward trajectory does not pass through $B\times
\{0\}$?
\end{question}

Now we consider the case when for a monotone equiconcave skew product
the base map $R$ is invertible and preserves an ergodic invariant
probability measure $\mu$ on $B$. This setup gives us more space for
maneuvers, because now the function $\phi$ has to be defined only
almost everywhere.

We construct first a pullback attractor. This was the method of
getting of invariant graph described for instance in~\cite{Keller}.
Namely, we set $\phi_n(\te)$ to be equal to the second component of
$F^n(\te,a)$. Since $F$ is monotone in the fibers, the sequence
$(\phi_n)_{n=0}^\infty$ is decreasing, and therefore convergent
pointwise on the whole $B$. Denote its limit by $\phi_K$. From the
definition it follows immediately that the graph of $\phi_K$ is
invariant. Therefore the set of points at which $\phi_K$ is 0, is also
invariant. By ergodicity of $\mu$, we see that either $\phi_K$ is 0
almost everywhere, or it is positive almost everywhere. However, note
that Example~\ref{noinvattr} shows that $\phi_K$ (which is 0 at all
points $\te_n$) may be not an attractor everywhere.

\begin{theorem}\label{phik}
Let the system $(X,F)$ with base $B$ and fiber space $[0,a]$ be a
monotone equiconcave skew product, and let the base map $R$ be
invertible and preserve an ergodic invariant probability measure $\mu$
on $B$. If the function $\phi_K$ is positive almost everywhere then
its graph is an attractor with the basin of attraction containing the
set $Z\times (0,a]$ for some set $Z\subset B$ of full measure $\mu$.
\end{theorem}

\begin{proof}
Assume that $\phi_K$ is positive almost everywhere. Let $Z$ be the set
of those points of $B$ at which $\phi_K$ is positive. This set is
invariant for $R$, so we can consider the restriction of $F$ to the
set $Z\times [0,a]$. This restriction satisfies the assumptions of
Corollary~\ref{nostructattr}, so the graph of $\phi_K$ restricted to
$Z$ is an attractor whose basin contains $Z\times (0,a]$. Since $Z$ is
of full measure, this completes the proof.
\end{proof}

\begin{question}
Does Theorem~\ref{phik} hold if we drop the assumption that $\phi_K$
is positive almost everywhere?
\end{question}



\end{document}